\documentclass[runningheads]{llncs}
\usepackage{graphicx}

\usepackage{lscape}
\usepackage{pdflscape}

\usepackage{stmaryrd}

\usepackage{cmll}
\usepackage[table]{xcolor}
\usepackage{comment}

\usepackage{mathdots}

\usepackage{amscd}
\usepackage{amssymb}
\usepackage{amsmath}
\usepackage{mathptmx}
\usepackage{mathrsfs}
\usepackage{color}
\usepackage{xspace}
\usepackage{bussproofs}
\EnableBpAbbreviations

\usepackage{tikz}

\usepackage{txfonts}

\usepackage{centernot}

\usepackage{mathtools}

\usepackage{relsize}

\usepackage{multirow}

\usepackage{multicol}

\usepackage{array}
\newcolumntype{C}[1]{>{\centering\arraybackslash}p{#1}}
\newcolumntype{L}[1]{>{\arraybackslash}p{#1}}

\usepackage{hyperref}
\usetikzlibrary{arrows}
\usetikzlibrary{matrix}
\usetikzlibrary{patterns}
\usetikzlibrary{shapes}

\newcommand{\val}[1]{[\![{#1}]\!]}
\newcommand{\descr}[1]{(\![{#1}]\!)}
\newcommand{\marginnote}[1]{\marginpar{\raggedright\tiny{#1}}} 
\newcommand{\nomi}{\mathbf{i}}
\newcommand{\nomj}{\mathbf{j}}

\newcommand{\cnomm}{\mathbf{m}}
\newcommand{\cnomn}{\mathbf{n}}

\newcommand{\Prop}{\mathsf{Prop}}
\newcommand{\X}{\mathbb{X}}
\newcommand{\z}{z}

\begin{document}
\title{Modelling informational entropy \thanks{The research of the third author is supported by the NWO Vidi grant 016.138.314, the NWO Aspasia grant 015.008.054, and a Delft
Technology Fellowship awarded in 2013.}}

%
%
\author{Willem Conradie\inst{1} \and
Andrew Craig\inst{2} \and
Alessandra Palmigiano\inst{2,3} \and
Nachoem M. Wijnberg\inst{2,4}}

\authorrunning{Conradie Craig Palmigiano Wijnberg}

\institute{University of the Witwatersrand, South Africa \and
University of Johannesburg, South Africa \and 
Delft University of Technology, The Netherlands \and 
University of Amsterdam, The Netherlands}

%
%

%
\maketitle              

\begin{abstract}
By `informational entropy', we understand an inherent boundary to knowability, due e.g.~to perceptual, theoretical, evidential or linguistic limits. In this paper, we discuss a logical framework in which this boundary is incorporated into the semantic and deductive machinery, and outline how this framework can be used to model various situations in which informational entropy arises. 

\keywords{Lattice-based modal logic \and Epistemic logic \and Concept lattice \and Graph-based semantics \and Polarity-based semantics}
\end{abstract}
\section{Introduction}
This paper contributes to a line of research stemming from the theory of canonicity and correspondence of lattice expansions \cite{gehrke2001bounded,conradie2016constructive,CoPa:non-dist,CCPZ}, which aims at 
defining and studying relational semantic frameworks for lattice-based logics. The present contribution specifically builds on the  {\em graph-based} semantics introduced in  \cite{conradie2015relational}, on the basis of a `modal expansion' of Plo\v{s}\v{c}ica's representation \cite{ploscica1994}, its relationship with canonical extensions of bounded lattices \cite{craig2013fresh,craig2015tirs}, and the ensuing algebraic canonicity and correspondence results \cite{conradie2015relational,CoPa:non-dist}. The resulting relational structures introduced in this paper, called {\em graph-based frames} (cf.~Definition \ref{def:graph:based:frame:and:model}), are more general than those in   \cite{conradie2015relational}, as the `TiRS' conditions have been removed. Hence, rather than being characterized as  discrete duals of perfect modal lattices, the graph-based structures considered here are in a discrete adjunction with complete modal lattices, much in the same way in which  the class of the relational structures interpreting the same logic in \cite{conradie2016categories}, which are based on polarities rather than on graphs, was generalized in \cite{Tarkpaper} so as to remove the `RS' conditions. However, the notions of satisfaction and refutation of formulas at states of graph-based frames can be extracted from their interpretation on the complex algebras of graph-based frames by an analogous `dual characterization' process which the frames-to-algebras direction of the adjunction is enough to convey.

Besides this technical contribution,  there is also a conceptual contribution which consists of making sense of this semantic framework in a more fundamental way. Our proposal in this respect is to use graph-based frames to provide a purely qualitative representation 
of the notion of   {\em relative entropy} in information theory \cite{Weaver-Shannon}, 
which is a stochastic measure of  {\em noise} in communication systems.  
As is argued by Weaver~\cite{Weaver-Shannon}, the significance of the key notions and insights developed in information theory goes very much beyond the original ``engineering aspects of communication'', and invests also such aspects as meaning and knowledge. If the notion of relative entropy is construed more broadly in this way, so as to
 capture  {\em conceptual noise}, then it can be understood as the inherent boundary to knowability due e.g.~to perceptual, theoretical, evidential or linguistic limits. In this paper, as specific examples, we model phenomena of informational entropy (under this broader understanding) arising in natural language and visual perception.  The interpretation proposed in the present paper  is further pursued in  \cite{graph-based-MV}, where informational entropy arises from the scientific theories on which empirical studies are grounded, and in \cite{tark1}, where it arises from socio-political theories.
 
Of course, the interpretation and use of graph-based structures proposed in the present paper  does not exclude the possibility of other interpretations and uses, as is suggested by the fact that the `companion' polarity-based semantics for lattice-based modal logic has been used to provide different interpretations of the lattice-based modal logic, including one in which lattice-based modal logic is viewed as an {\em epistemic logic of categories} \cite{conradie2016categories,Tarkpaper} and one \cite{roughconcepts,ICLA2019paper} in which the same logic is viewed as the {\em logic of rough concepts}, where polarity-based semantics  is used as an encompassing framework for the integration of rough set theory \cite{pawlak} and formal concept analysis \cite{ganter2012formal}, and as a basis for further developments such as a 
Dempster--Shafer theory of concepts \cite{TarkDS}.

\section{Preliminaries}
\paragraph{Notation.} 
We let $\Delta_U$ denote the identity relation on a set $U$, and we will drop the subscript when it causes no ambiguity. The superscript $(\cdot)^c$ denotes the relative complement of the subset of a given set. Hence, for any binary relation $R\subseteq S\times T$, we let $R^c\subseteq S\times T$ be defined by  $(s, t)\in R^c$ iff $(s, t)\notin R$.  For any such $R$ and any $S'\subseteq S$ and $T'\subseteq T$, we  let $R[S']: = \{t\in T\mid (s, t)\in R \mbox{ for some } s\in S'\}$ and $R^{-1}[T']: = \{s\in S\mid (s, t)\in R \mbox{ for some } t\in T'\}$, and write $R[s]$ and $R^{-1}[t]$ for $R[\{s\}]$ and $R^{-1}[\{t\}]$, respectively. Any such $R$ gives rise to the
  {\em semantic modal operators} $\langle R\rangle, [R]: \mathcal{P}(T)\to \mathcal{P}(S)$ s.t.~
  $\langle R\rangle W : = R^{-1}[W] $ and $ [R]W: = (R^{-1}[W^c])^c$ for any $W\subseteq T$.
For any  $T\subseteq U\times V$, and any $U'\subseteq U$  and $V'\subseteq V$, let
\begin{equation}\label{eq:def;round brackets} T^{(1)}[U']:=\{v\mid \forall u(u\in U'\Rightarrow uTv) \}  \quad\quad T^{(0)}[V']:=\{u\mid \forall v(v\in V'\Rightarrow uTv) \}.\end{equation}
Known properties of this construction (cf.~\cite[Sections 7.22-7.29]{davey2002introduction}) are collected below.
 \begin{lemma}\label{lemma: basic}
\begin{enumerate}
\item $X_1\subseteq X_2\subseteq U$ implies $T^{(1)}[X_2]\subseteq T^{(1)}[X_1]$, and $Y_1\subseteq Y_2\subseteq V$ implies $T^{(0)}[Y_2]\subseteq T^{(0)}[Y_1]$.
\item $U'\subseteq T^{(0)}[V']$ iff  $V'\subseteq T^{(1)}[U']$.
 \item $U'\subseteq T^{(0)}[T^{(1)}[U']]$ and $V'\subseteq T^{(1)}[T^{(0)}[V']]$.
 \item $T^{(1)}[U'] = T^{(1)}[T^{(0)}[T^{(1)}[U']]]$ and $T^{(0)}[V'] = T^{(0)}[T^{(1)}[T^{(0)}[V']]]$.
 \item $T^{(0)}[\bigcup\mathcal{V}] = \bigcap_{V'\in \mathcal{V}}T^{(0)}[V']$ and $T^{(1)}[\bigcup\mathcal{U}] = \bigcap_{U'\in \mathcal{U}}T^{(1)}[U']$.
\end{enumerate}
 \end{lemma}
For any relation $T\subseteq U\times V$, and any $U'\subseteq U$  and $V'\subseteq V$, let
\begin{equation}\label{eq:def:square brackets}T^{[1]}[U']:=\{v\mid \forall u(u\in U'\Rightarrow uT^cv) \}  \quad\quad T^{[0]}[V']:=\{u\mid \forall v(v\in V'\Rightarrow uT^cv) \}.\end{equation}
Hence, $T^{[1]}[U'] = (T^c)^{(1)}[U']$ and $T^{[0]}[V'] = (T^c)^{(0)}[V']$, therefore, the following lemma is an immediate consequence of 
 Lemma \ref{lemma: basic} instantiated to $T: = T^c$.
 \begin{lemma}\label{lemma: basic:square:brackets}
\begin{enumerate}
\item $X_1\subseteq X_2\subseteq U$ implies $T^{[1]}[X_2]\subseteq T^{[1]}[X_1]$, and $Y_1\subseteq Y_2\subseteq V$ implies $T^{[0]}[Y_2]\subseteq T^{[0]}[Y_1]$.
\item $U'\subseteq T^{[0]}[V']$ iff  $V'\subseteq T^{[1]}[U']$.
 \item $U'\subseteq T^{[0]}[T^{[1]}[U']]$ and $V'\subseteq T^{[1]}[T^{[0]}[V']]$.
 \item $T^{[1]}[U'] = T^{[1]}[T^{[0]}[T^{[1]}[U']]]$ and $T^{[0]}[V'] = T^{[0]}[T^{[1]}[T^{[0]}[V']]]$.
 \item $T^{[0]}[\bigcup\mathcal{V}] = \bigcap_{V'\in \mathcal{V}}T^{[0]}[V']$ and $T^{[1]}[\bigcup\mathcal{U}] = \bigcap_{U'\in \mathcal{U}}T^{[1]}[U']$.
\end{enumerate}
 \end{lemma}
\subsection{Basic normal non-distributive modal logic}
\label{sec:logics}
The logic discussed below was considered in \cite{conradie2016categories} as an instance of a logic to which a general methodology applies  for endowing lattice-based logics with relational semantics (cf.~\cite[Section 2]{CoPa:non-dist}). The semantics of this logic was  based on a restricted class of formal contexts. These restrictions were lifted in \cite{Tarkpaper}. 
\paragraph{Basic logic.} Let $\Prop$ be a (countable or finite) set of atomic propositions. The language $\mathcal{L}$ of the {\em basic normal non-distributive modal logic} is defined as follows:
\[ \varphi := \bot \mid \top \mid p \mid  \varphi \wedge \varphi \mid \varphi \vee \varphi \mid \Box \varphi \mid  \Diamond\varphi,\] 
where $p\in \Prop$. 
The {\em basic}, or {\em minimal normal} $\mathcal{L}$-{\em logic} is a set $\mathbf{L}$ of sequents $\phi\vdash\psi$  with $\phi,\psi\in\mathcal{L}$, containing the following axioms:
		{\small{
			\begin{align*}
				&p\vdash p, && \bot\vdash p, && p\vdash \top, & &  &\\
				&p\vdash p\vee q, && q\vdash p\vee q, && p\wedge q\vdash p, && p\wedge q\vdash q, &\\
				& \top\vdash \Box \top, && \Box p\wedge \Box q \vdash \Box ( p\wedge q),
           && \Diamond \bot\vdash \bot, &&
                \Diamond p \vee \Diamond  q \vdash \Diamond ( p \vee q) &
			\end{align*}
			}}
	%
	%
		and closed under the following inference rules:
		{\small{
		\begin{displaymath}
			\frac{\phi\vdash \chi\quad \chi\vdash \psi}{\phi\vdash \psi}
			\quad
			\frac{\phi\vdash \psi}{\phi\left(\chi/p\right)\vdash\psi\left(\chi/p\right)}
			\quad
			\frac{\chi\vdash\phi\quad \chi\vdash\psi}{\chi\vdash \phi\wedge\psi}
			\quad
			\frac{\phi\vdash\chi\quad \psi\vdash\chi}{\phi\vee\psi\vdash\chi}
			\quad
			\frac{\phi\vdash\psi}{\Box \phi\vdash \Box \psi}
\quad
\frac{\phi\vdash\psi}{\Diamond \phi\vdash \Diamond \psi}
		\end{displaymath}
		}}
By an {\em $\mathcal{L}$-logic} we understand any  extension of $\mathbf{L}$  with $\mathcal{L}$-axioms $\phi\vdash\psi$.

\paragraph{Algebraic semantics.} 
The logic above is sound and complete w.r.t.~the class $\mathbb{LE}$ of normal lattice expansions $\mathbb{A} = (\mathbb{L}, \Box, \Diamond)$, where $\mathbb{L} =(L, \wedge, \vee, \top, \bot)$ is a general lattice, and  $\Box$ and $\Diamond$ are unary operations on $\mathbb{L}$ satisfying the following identities:
\[\Box \top = \top \quad \Box(a\wedge b) = \Box a\wedge \Box b\quad \Diamond\bot = \bot\quad \Diamond(a\vee b) = \Diamond a\vee\Diamond b.\]
In what follows, we will sometimes refer to elements of $\mathbb{LE}$ as $\mathcal{L}$-algebras.
Since  $\mathbf{L}$ is selfextensional (i.e.\ the interderivability relation is a congruence of the formula algebra),  a standard 
Lindenbaum--Tarski construction is sufficient to show its completeness w.r.t.~$\mathbb{LE}$, i.e.~that an $\mathcal{L}$-sequent $\phi\vdash\psi$ is in $\mathbf{L}$ iff $\mathbb{LE}\models\phi\vdash\psi$.

\section{Graph-based semantics for the basic non-distributive modal logic}
\label{sec:graph-based L-frames and models}

Graph-based models for non-distributive logics arise in close connection with the topological structures dual to general lattices in Plo\v{s}\v{c}ica's representation \cite{ploscica1994}, see also \cite{craig2013fresh,craig2015tirs}. However, an important difference in the current paper is that we do not require the TiRS conditions \cite[Section 2]{craig2015tirs}.  

A {\em reflexive graph} is a structure $\X = (Z, E)$ such that $Z$ is a nonempty set, and $E\subseteq Z\times Z$ is a reflexive relation. From now on, we will assume that all graphs we consider are reflexive even when we drop the adjective.  
Any graph $\X = (Z, E)$  defines the polarity\footnote{ A {\em formal context} \cite{ganter2012formal}, or {\em polarity},  is a structure $\mathbb{P} = (A, X, I)$ such that $A$ and $X$ are sets, and $I\subseteq A\times X$ is a binary relation. Every such $\mathbb{P}$ induces maps $(\cdot)^\uparrow: \mathcal{P}(A)\to \mathcal{P}(X)$ and $(\cdot)^\downarrow: \mathcal{P}(X)\to \mathcal{P}(A)$, respectively defined by the assignments $B^\uparrow: = I^{(1)}[B]$ and $Y^\downarrow: = I^{(0)}[Y]$. A {\em formal concept} of $\mathbb{P}$ is a pair $c = (B, Y)$ such that $B\subseteq A$, $Y\subseteq X$, and $B^{\uparrow} = Y$ and $Y^{\downarrow} = B$. Given a formal concept $c = (B,Y)$ we will often write $\val{c}$ for $B$ and $\descr{c}$ for $Y$ and, consequently, $c = (\val{c}, \descr{c})$.  The set $L(\mathbb{P})$  of the formal concepts of $\mathbb{P}$ can be partially ordered as follows: for any $c = (B_1, Y_1), d = (B_2, Y_2)\in L(\mathbb{P})$, \[c\leq d\quad \mbox{ iff }\quad B_1\subseteq B_2 \quad \mbox{ iff }\quad Y_2\subseteq Y_1.\]
	With this order, $L(\mathbb{P})$ is a complete lattice, the {\em concept lattice} $\mathbb{P}^+$ of $\mathbb{P}$. Any complete lattice $\mathbb{L}$ is isomorphic to the concept lattice $\mathbb{P}^+$ of some polarity $\mathbb{P}$.} $\mathbb{P_X} = (Z_A,Z_X, I_{E^{c}})$ where $Z_A = Z = Z_X$ and  $I_{E^{c}}\subseteq Z_A\times Z_X$ is defined as $aI_{E^{c}} x$ iff $aE^cx$.  More generally, any relation $R\subseteq Z\times Z$ `lifts' to relations $I_{R^c}\subseteq Z_A\times Z_X$ and $J_{R^c}\subseteq Z_X\times Z_A$
defined as $aI_{R^c} x$ iff $a R^c x$ and $xJ_{R^c} a$ iff $x R^c a$.
The next lemma follows directly from these definitions:
\begin{lemma}
	\label{lemma:round and square brackets}
	For any relation $R\subseteq Z \times Z$ and any $Y, B \subseteq Z$,
	\[I^{(0)}_{R^c}[Y] = R^{[0]}[Y]\quad  I^{(1)}_{R^c}[B] = R^{[1]}[B]\quad J^{(0)}_{R^c}[B] = R^{[0]}[B]\quad  J^{(1)}_{R^c}[Y] = R^{[1]}[Y].\]
\end{lemma}
The complete lattice $\X^{+}$ associated with a graph $\X$ is defined as the concept lattice of $\mathbb{P_X}$.
For any lattice $\mathbb{L}$, let $\mathsf{Flt}(\mathbb{L})$ and $\mathsf{Idl}(\mathbb{L})$ denote the set of filters and ideals of $\mathbb{L}$, respectively. The graph associated with $\mathbb{L}$ is $\mathbb{X_L}:=(Z,E)$ where $Z$ is the set of tuples
$(F, J)\in \mathsf{Flt}(\mathbb{L})\times \mathsf{Idl}(\mathbb{L})$ such that  $F\cap J = \varnothing$. For $z \in Z$, we denote by $F_z$ the filter part of $z$ and by $J_z$ the ideal part of $z$. Clearly, filter parts and ideal parts of states of $\mathbb{X_L}$ must be proper.  The (reflexive) $E$ relation is defined by 
$zEz'$ if and only if $F_z \cap J_{z'}= \emptyset$. 

\begin{definition}\label{Def:Canon:Ext}\cite[Section 2]{gehrke2001bounded}
Let $\mathbb{L}$ be a (bounded) sublattice of a complete lattice $\mathbb{L}'$.
\begin{enumerate}
\item $\mathbb{L}$ is {\em dense} in\, $\mathbb{L}'$ if every element of\, $\mathbb{L}'$ can be expressed both as a join of meets and as a meet of joins of elements from\, $\mathbb{L}$.
\item $\mathbb{L}$ is {\em compact} in\, $\mathbb{L}'$ if, for all $S, T \subseteq L$, if $\bigvee S\leq \bigwedge T$ then $\bigvee S'\leq \bigwedge T'$ for some finite $S'\subseteq S$ and $T'\subseteq T$.
\item The {\em canonical extension} of a lattice\, $\mathbb{L}$ is a complete lattice\, $\mathbb{L}^{\delta}$ containing\, $\mathbb{L}$ as a dense and compact sublattice.
\end{enumerate}
\end{definition}
The canonical extension of any bounded lattice exists~\cite[Proposition 2.6]{gehrke2001bounded} and is unique up to isomorphism~\cite[Proposition 2.7]{gehrke2001bounded}.

\begin{proposition}\label{Prop:Xplus:canonicalextension}\cite[Proposition 4.2]{craighaviar2014recon}
For any lattice $\mathbb{L}$, the complete lattice $\mathbb{X_L}^+$ is the canonical extension of\, $\mathbb{L}$. 
\end{proposition}

Furthermore, from results in~\cite[Sections 5 and 6]{gehrke2001bounded}, we know that if 
$\mathbb{A}=(\mathbb{L},\Box,\Diamond)$ is an $\mathcal{L}$-algebra, then the additional operations can be extended to 
$\mathbb{X_L}^+$ in order to get a complete $\mathcal{L}$-algebra. 
%
%

\begin{definition}\label{def:graph:based:frame:and:model}
	A {\em graph-based} $\mathcal{L}$-\emph{frame}  is a structure $\mathbb{F} = (\mathbb{X},  R_{\Diamond}, R_{\Box})$ where $\mathbb{X} = (Z,E)$ is a reflexive graph,\footnote{\label{footnote: abbreviations} Applying the notation \eqref{eq:def:square brackets} to a graph-based $\mathcal{L}$-frame $\mathbb{F}$, we will sometimes abbreviate $E^{[0]}[Y]$ and $E^{[1]}[B]$ as $Y^{[0]}$ and $B^{[1]}$, respectively, for each $Y, B\subseteq Z$. If $Y= \{y\}$ and  $B = \{b\}$, we write $y^{[0]}$ and $b^{[1]}$ for $\{y\}^{[0]}$ and $\{b\}^{[1]}$, and write $Y^{[01]}$ and $B^{[10]}$ for $(Y^{[0]})^{[1]}$ and $(B^{[1]})^{[0]}$, respectively. Notice that, by Lemma \ref{lemma:round and square brackets}, $Y^{[0]} = I_{E^c}^{(0)}[Y] = Y^{\downarrow}$ and $B^{[1]} = I_{E^c}^{(1)}[B] = B^{\uparrow}$, where the maps $(\cdot)^\downarrow$ and $(\cdot)^\uparrow$ are those associated with the polarity $\mathbb{P_X}$.} and  $R_{\Diamond}$ and $R_{\Box}$  are binary relations on $Z$ satisfying the following  $E$-{\em compatibility} conditions (notation defined in \eqref{eq:def:square brackets}): for all $b,y \in Z$,
\begin{align*}
(R_\Box^{[0]}[y])^{[10]} &\subseteq R_\Box^{[0]}[y] \qquad   &(R_\Box^{[1]}[b])^{[01]} \subseteq R_\Box^{[1]}[b]\\
(R_\Diamond^{[0]}[b])^{[10]} &\subseteq R_\Diamond^{[0]}[b] \qquad   &(R_\Diamond^{[1]}[y])^{[01]} \subseteq R_\Diamond^{[1]}[y].
\end{align*}	
	
	The {\em complex algebra} of a graph-based $\mathcal{L}$-frame $\mathbb{F}= (\mathbb{X},  R_{\Diamond}, R_{\Box})$ is the complete $\mathcal{L}$-algebra $\mathbb{F}^+ = (\mathbb{X}^+, [R_\Box], \langle R_\Diamond\rangle),$
	where $\mathbb{X}^+$ is the concept lattice of 
	$\mathbb{P}_{\mathbb{X}}$,
	and $[R_\Box]$ and $\langle R_\Diamond\rangle$ are unary operations on 
	$\mathbb{P}_{\mathbb{X}}^+$
	defined as follows: for every $c = (\val{c}, \descr{c}) \in 
	\mathbb{P}_{\mathbb{X}}^+$,
		\[[R_\Box]c: = (R_{\Box}^{[0]}[\descr{c}], (R_{\Box}^{[0]}[\descr{c}])^{[1]}) \quad \mbox{ and }\quad \langle R_\Diamond\rangle c: = ((R_{\Diamond}^{[0]}[\val{c}])^{[0]}, R_{\Diamond}^{[0]}[\val{c}]).\]
	%
	
\end{definition}

The following lemma is an immediate consequence of Lemma~\ref{equivalents of I-compatible appendix} in the appendix, using Lemma~\ref{lemma:round and square brackets} and the observation in Footnote \ref{footnote: abbreviations}. 
\begin{lemma}\label{equivalents of I-compatible}
	\begin{enumerate}
		\item The following are equivalent for every graph $(Z,E)$ and every relation $R\subseteq Z\times Z$:
		\begin{enumerate}
			\item [(i)] $(R^{[0]}[y])^{[10]} \subseteq R^{[0]}[y]$  for every $y\in Z$;
			\item [(ii)]  $(R^{[0]}[Y])^{[10]} \subseteq R^{[0]}[Y]$ for every $Y\subseteq Z$;
			\item [(iii)] $R^{[1]}[B]=R^{[1]}[B^{[10]}]$ for every  $B\subseteq Z$.
		\end{enumerate}
		\item The following are equivalent for every  graph $(Z,E)$ and every relation $R\subseteq Z\times Z$:
		\begin{enumerate}
			\item [(i)] $(R^{[1]}[b])^{[01]} \subseteq R^{[1]}[b]$  for every $b\in Z$;
			\item [(ii)]  $(R^{[1]}[B])^{[01]} \subseteq R^{[1]}[B]$ for every $B\subseteq Z$;
			\item [(iii)] $R^{[0]}[Y]=R^{[0]}[Y^{[01]}]$ for every  $Y\subseteq Z$.
		\end{enumerate}
	\end{enumerate}
\end{lemma}
For any graph-based $\mathcal{L}$-frame $\mathbb{F}$, let us define $R_{\Diamondblack}\subseteq Z\times Z$ by $x R_{\Diamondblack} a$ iff $aR_{\Box} x$, and $R_{\blacksquare}\subseteq Z\times Z$ by $a R_{\blacksquare} x$ iff $xR_{\Diamond} a$. Hence, for every $B, Y\subseteq Z$,
\begin{equation}
\label{eq:zero is 1 is zero}
R_{\Diamondblack}^{[0]}[B] = R_{\Box}^{[1]}[B] \quad R_{\Diamondblack}^{[1]}[Y] = R_{\Box}^{[0]}[Y] \quad R_{\blacksquare}^{[0]}[Y] = R_{\Diamond}^{[1]}[Y] \quad R_{\blacksquare}^{[1]}[B] = R_{\Diamond}^{[0]}[B].
\end{equation}
By Lemma \ref{equivalents of I-compatible}, the $E$-compatibility of $R_{\Box}$ and $ R_{\Diamond}$ guarantees that  the operations $[R_\Box], \langle R_\Diamond\rangle$ (as well as $[R_\blacksquare], \langle R_\Diamondblack\rangle$) are well defined on 
$\mathbb{X}^+$. 

\begin{lemma}
Let \,  $\mathbb{F} = (\mathbb{X}, R_{\Box}, R_{\Diamond})$ be a graph-based $\mathcal{L}$-frame. Then  
	the algebra\, $\mathbb{F}^+ = (\mathbb{X}^+, [R_{\Box}], \langle R_{\Diamond}\rangle)$ is a complete  lattice expansion such that $[R_\Box]$ is completely meet-preserving and $\langle R_\Diamond\rangle$ is completely join-preserving.
\end{lemma}

\begin{proof}
	As mentioned above, the $E$-compatibility of $R_{\Box}$ and $R_{\Diamond}$ guarantees that the maps 
	$[R_\Box],\langle R_{\Diamond}\rangle, [R_\blacksquare],\langle R_{\Diamondblack}\rangle: \mathbb{X}^+\to \mathbb{X}^+$ 
	are well defined. Since 
	$\mathbb{X}^{+}$ 
	is a complete lattice, by \cite[Proposition 7.31]{davey2002introduction}, to show that $[R_\Box]$ is completely meet-preserving and $\langle R_\Diamond\rangle$ is completely join-preserving, it is enough to show that  $\langle R_{\Diamondblack}\rangle$ is the left adjoint of $[R_{\Box}]$ and $[R_\blacksquare]$ is the right adjoint of $\langle R_\Diamond\rangle$. For any $c = (\val{c}, \descr{c}),d= (\val{d}, \descr{d})\in \mathbb{X}^{+}$,
	{\small{
			\begin{center}
				\begin{tabular}{r c l l}
					$\langle R_{\Diamondblack}\rangle c\le d $ &
					iff & $ \descr{d} \subseteq R_{\Diamondblack}^{[0]}[\val{c}] $ & ordering of concepts\\
					& iff & $ \descr{d} \subseteq R_{\Box}^{[1]}[\val{c}] $   & \eqref{eq:zero is 1 is zero}  \\
					&iff& $ \val{c} \subseteq R_{\Box}^{[0]}[\descr{d}] $& Lemma \ref{lemma: basic:square:brackets}.2\\
					&iff& $ c\le [R_{\Box}]d .$& ordering of concepts
					
				\end{tabular}
			\end{center}
	}}
	Likewise, one shows that $[R_\blacksquare]$ is the right adjoint of $\langle R_\Diamond\rangle$.
\end{proof}


For an $\mathcal{L}$-algebra $\mathbb{L}$ and $K \subseteq L$, we let 
$$\Box K = \{\, \Box u \mid u \in K\,\} \qquad \text{and} \qquad \Diamond K = \{\, \Diamond v \mid v \in K\,\}.$$ 
Further, for $K \subseteq L$, we denote by $\lceil K \rceil$ ($\lfloor K \rfloor$) the ideal (filter) generated by $K$. 

\begin{lemma} \label{lemma:filtidl:gen:intersect} Let\, $\mathbb{L}$ be an $\mathcal{L}$-algebra with $F \in \mathsf{Flt}(\mathbb{L})$ and $J \in \mathsf{Idl}(\mathbb{L})$. Then 
\begin{enumerate}
\item $F \cap \Box J \neq \emptyset$ \, if and only if \, $F \cap \lceil \Box J \rceil \neq \emptyset$; 
\item $\Diamond F \cap  J \neq \emptyset$ \, if and only if \, $\lfloor \Diamond F \rfloor \cap J \neq \emptyset$. 
\end{enumerate}
\end{lemma}
\begin{proof}
Let us prove item 1. The left-to-right direction is immediate since $\Box J\subseteq \lceil \Box J \rceil$. Conversely, assume that  there are elements $u_1,\ldots,u_n \in J$ such that $ \Box u_1\vee \cdots \vee \Box u_n \in F$. Because $\Box$ is monotone and $F$ is upward closed, then  $\Box( u_1\vee \cdots\vee  u_n) \in F$. Because $u_1,\ldots,u_n \in J$ and $J$ is an ideal, then  $u_1\vee \cdots \vee  u_n\in J$, which completes the proof that $F \cap \Box J \neq \emptyset$. The proof of item 2 is similar and omitted.
\end{proof}

\begin{definition}
	Given a complete $\mathcal{L}$-algebra $\mathbb{A} = (\mathbb{L}, \Box, \Diamond)$ we define its associated $\mathcal{L}$-frame to be the structure $\mathbb{F_A} = (\mathbb{X_L}, R_{\Box}, R_{\Diamond})$ where 
	$R_{\Box}, R_{\Diamond} \subseteq Z \times Z$ 
	are given by $x R_{\Box}y$ iff $F_x \cap \Box J_y = \emptyset$ 
	and $xR_{\Diamond}y$ iff $J_x \cap \Diamond F_y = \emptyset$. 	
\end{definition}
\begin{proposition}\label{prop:alg-to-frame}
	For any $\mathcal{L}$-algebra $\mathbb{A}$, the associated $\mathcal{L}$-frame $\mathbb{F_A}$ is a graph-based $\mathcal{L}$-frame.
\end{proposition}

\begin{proof} We show that $(R_{\Box}^{[0]}[y])^{[10]} \subseteq R_{\Box}^{[0]}[y]$. The other three properties will follow by similar arguments. 
With the help of Lemma~\ref{lemma:filtidl:gen:intersect}(1), we observe that
$$R_{\Box}^{[0]}[y] = \{\, u \in Z \mid  (u, y)\notin R_{\Box} \,\}= \{\, u \in Z \mid F_u \cap \Box J_y \neq \emptyset\,\} =\{\, u \in Z \mid F_u \cap \lceil \Box J_y 
\rceil \neq \emptyset\,\}.$$
We have 
\begin{align*}
(R_{\Box}^{[0]}[y])^{[1]}  &= \{\, z \in Z \mid \forall u (u \in R_{\Box}^{[0]}[y] \Rightarrow (u,z) \notin E)\,\} \\
&= \{\, z \in Z \mid \forall u( F_u \cap \Box J_y \neq \emptyset \Rightarrow F_u \cap J_z \neq \emptyset)\,\}\\
&= \{\, z \in Z \mid \Box J_y \subseteq J_z \,\}. 
\end{align*}
Hence 
\begin{center}
\begin{tabular}{rcll}
$a \in (R_{\Box}^{[0]}[y])^{[10]}$ & iff & $\forall z \in Z (\Box J_y \subseteq J_z \Rightarrow (a,z) \notin E )$ \\
& iff & $\forall z \in Z (\Box J_y \subseteq J_z \Rightarrow F_a \cap J_z \neq\emptyset)$ \\
& iff & $F_a \cap \lceil \Box J_y \rceil \neq \emptyset$ \\
& iff & $F_a \cap \Box J_y \neq \emptyset$ & Lemma \ref{lemma:filtidl:gen:intersect}\\
& iff & $a\in R_{\Box}^{[0]}[y]$. 
\end{tabular}
\end{center}

\end{proof}

%

\begin{definition}
	A {\em graph-based} $\mathcal{L}$-\emph{model} is a tuple $\mathbb{M} = (\mathbb{F}, V)$ where $\mathbb{F}$ is a graph-based $\mathcal{L}$-frame and $V: \mathsf{Prop}\to \mathbb{F}^+$. Since $V(p)$ is therefore a formal concept, we will write $V(p) = (\val{p}, \descr{p})$.
	%
	%
	%
	%
\end{definition}
For every graph-based $\mathcal{L}$-model $\mathbb{M} = (\mathbb{F}, V)$,  the valuation $V$ can be extended compositionally to all $\mathcal{L}$-formulas as follows:
{\small{
		\begin{center}
			\begin{tabular}{r c l c r cl}
				$V(p)$ & $ = $ & $(\val{p}, \descr{p})$\\
				$V(\top)$ & $ = $ & $(Z, \emptyset)$ &\quad& $V(\bot)$ & $ = $ & $(\emptyset, Z)$\\
				$V(\phi\wedge\psi)$ & $ = $ & $(\val{\phi}\cap \val{\psi}, (\val{\phi}\cap \val{\psi})^{[1]})$ &&
				$V(\phi\vee\psi)$ & $ = $ & $((\descr{\phi}\cap \descr{\psi})^{[0]}, \descr{\phi}\cap \descr{\psi})$\\
				$V(\Box\phi)$ & $ = $ & $(R_{\Box}^{[0]}[\descr{\phi}], (R_{\Box}^{[0]}[\descr{\phi}])^{[1]})$ &&
				$V(\Diamond\phi)$ & $ = $ & $((R_{\Diamond}^{[0]}[\val{\phi}])^{[0]}, R_{\Diamond}^{[0]}[\val{\phi}])$\\
			\end{tabular}
		\end{center}
}}
and moreover, the existence of the adjoints of $[R_{\Box}]$ and $\langle R_\Diamond \rangle$ supports the interpretation of the following expansion:
{\small{
		\begin{center}
			\begin{tabular}{r c l c r c l}
				$V(\blacksquare\phi)$ & $ = $ & $(R_{\blacksquare}^{[0]}[\descr{\phi}], (R_{\blacksquare}^{[0]}[\descr{\phi}])^{[1]})$&$\quad$&
				$V(\Diamondblack\phi)$ & $ = $ & $((R_{\Diamondblack}^{[0]}[\val{\phi}])^{[0]}, R_{\Diamondblack}^{[0]}[\val{\phi}])$\\
				%
			\end{tabular}
\end{center}}}
Spelling out the definition above (cf.~\cite{conradie2015relational}), we can define the satisfaction and co-satisfaction relations $\mathbb{M}, \z \Vdash \phi$ and $\mathbb{M}, \z \succ \phi$
for every graph-based $\mathcal{L}$-model $\mathbb{M} = (\mathbb{F}, V)$, $\z \in Z$, and any $\mathcal{L}$-formula $\phi$,  by the following simultaneous recursion:
{\small{
		\begin{flushleft}
			\begin{tabular}{lllllll}
				$\mathbb{M}, \z \Vdash \bot$ &&never & &$\mathbb{M}, \z \succ \bot$ && always\\
				$\mathbb{M}, \z \Vdash \top$ &&always & &$\mathbb{M}, \z \succ \top$ &&never\\
				$\mathbb{M}, \z \Vdash p$ & iff & $\z\in \val{p}$ & &$\mathbb{M}, \z \succ p$ & iff & $\forall \z'[\z'E\z \Rightarrow \z' \not\Vdash p]$\\
				%
				%
				%
				$\mathbb{M}, \z \succ \phi \vee \psi$ &iff &$\mathbb{M},\z\succ \phi \text{ and } \mathbb{M},\z \succ \psi$%
				& &$\mathbb{M}, \z \Vdash \phi \vee \psi$ &iff &$\forall \z' [\z E\z'   \Rightarrow \mathbb{M},\z' \not\succ \phi \vee \psi]$\\
				$\mathbb{M}, \z \Vdash \phi \wedge \psi$ &iff &$\mathbb{M},\z\Vdash \phi \text{ and } \mathbb{M},\z \Vdash \psi$%
				& &$\mathbb{M},\z \succ \phi \wedge\psi$ &iff &$\forall \z'[\z' E\z \Rightarrow \mathbb{M},\z'\not\Vdash \phi \wedge \psi]$\\
				$\mathbb{M}, \z \succ \Diamond\phi$ &iff &$\forall \z' [\z R_{\Diamond}\z'  \Rightarrow \mathbb{M},\z' \not \Vdash \phi]$%
				& &$\mathbb{M}, \z \Vdash \Diamond\phi$ &iff &$\forall \z'[\z E\z' \Rightarrow \mathbb{M},\z' \not\succ \Diamond \phi]$\\
				$\mathbb{M}, \z \Vdash \Box \psi$ &iff &$\forall \z' [\z R_{\Box}\z'  \Rightarrow \mathbb{M},\z' \not\succ \psi]$%
				& &$\mathbb{M}, \z \succ \Box\psi$ &iff &$\forall \z' [\z' E\z  \Rightarrow \mathbb{M},\z' \not\Vdash \Box \psi]$\\
				%
				%
			\end{tabular}
		\end{flushleft}
}}
An $\mathcal{L}$-sequent $\phi \vdash \psi$ is \emph{true} in $\mathbb{M}$, denoted $\mathbb{M} \models \phi \vdash \psi$, if for all $\z, \z' \in Z$, if $\mathbb{M}, \z \Vdash \phi$ and $\mathbb{M}, \z' \succ \psi$ then $\z E^c z'$. An $\mathcal{L}$-sequent $\phi \vdash \psi$ is \emph{valid} in $\mathbb{F}$, denoted $\mathbb{F} \models \phi \vdash \psi$, if it is true in every model based on $\mathbb{F}$.

The next lemma follows immediately from the definition of an $\mathcal{L}$-sequent being true in a graph-based $\mathcal{L}$-model. 
\begin{lemma}
	\label{lemma:for:completeness}
	Let $\mathbb{F}$ be a graph-based $\mathcal{L}$-frame and $\phi\vdash \psi$ an $\mathcal{L}$-sequent. Then $\mathbb{F}\models \phi\vdash \psi$ \, iff \,  $\mathbb{F}^+ \models \phi\vdash \psi$. 
\end{lemma}

The next proposition follows from the fact that $\mathbf{L}$ is sound and complete with respect to the class of $\mathcal{L}$-algebras and Lemma~\ref{lemma:for:completeness}.
\begin{proposition}
	The basic non-distributive modal logic $\mathbf{L}$ is sound w.r.t.~the class of graph-based $\mathcal{L}$-frames. I.e., if an  $\mathcal{L}$-sequent $\phi\vdash \psi$ is provable in  $\mathbf{L}$, then $\mathbb{F}\models \phi\vdash \psi$ for every graph-based frame $\mathbb{F}$.
\end{proposition}

Let $\mathbb{A}_{\mathbf{L}}$ be the Lindenbaum--Tarski algebra of $\mathbf{L}$.  We will abuse notation and write $\phi$ instead $[\phi]$ (i.e.\ formulas instead of their equivalence classes) for the elements of  the Lindenbaum--Tarski algebra $\mathbb{A}_{\mathbf{L}}$. Define the {\em canonical graph-based model} to be 
$\mathbb{M}_{\mathbf{L}}=(\mathbb{F}_{\mathbb{A}_\mathbf{L}}, V)$ where $V(p)=(\{\,z\in Z \mid p \in F_z\}, \{\, z \in Z \mid p \in J_z \,\})$. 
By Proposition \ref{prop:alg-to-frame}, $\mathbb{F}_{\mathbb{A}_\mathbf{L}}$ is a graph-based $\mathcal{L}$-frame. That $V$ is well defined can be shown as follows:
\begin{center}
\begin{tabular}{r c l}
$(\{z \in Z \mid p \in F_z\})^{[1]}$ &  =& $ \{z \in Z \mid \forall z'(p \in F_{z}\Rightarrow (z, z')\notin E)\}$\\
&  = & $\{z \in Z \mid \forall z'( p \in F_{z}\Rightarrow  F_z\cap J_{z'}\neq \emptyset)\} $\\
& =& $\{z \in Z \mid p \in  J_z \}$
\end{tabular}
\end{center}
\begin{lemma}\label{lemma:truthlemma}
Let $\phi \in \mathcal{L}$. Then 
\begin{enumerate}
\item $\mathbb{M}_{\mathbf{L}}, z \Vdash \phi$ \, iff \, $\phi \in F_z$
\item $\mathbb{M}_{\mathbf{L}}, z \succ  \phi$ \, iff \, $\phi \in J_z$
\end{enumerate} 
\end{lemma}
\begin{proof}

Let us show item 1 under the additional assumption that  $\phi$ is  a theorem of $\mathbf{L}$ (i.e.~$\mathbf{L}$ derives $\top \vdash \phi$). Then $\phi$ belongs to every filter, hence to show the required equivalence, we need to show that $\val{\phi}_{\mathbb{M}_{\mathbf{L}}} = Z$.  If $\mathbf{L}$ derives $\top \vdash \phi$, then, by soundness, $\mathbb{M}_{\mathbf{L}}\models \top\vdash \phi$. Then  for every state $z$ in $\mathbb{M}_{\mathbf{L}}$, we have $\mathbb{M}_{\mathbf{L}}, z\not \succ \phi$. Indeed, suppose for contradiction that $\mathbb{M}_{\mathbf{L}}, z \succ \phi$ for some state $z$. Since $\mathbb{M}_{\mathbf{L}}, z \Vdash \top$, then by spelling out  the definition of satisfaction of a sequent in a model in the instance $\mathbb{M}_{\mathbf{L}}\models \top\vdash \phi$, we would conclude that $(z, z)\notin E$, i.e.~$E$ is not reflexive, which contradicts the fact that $E$ is reflexive by construction. This finishes the proof that if $\mathbf{L}$ derives $\top \vdash \phi$, then $\descr{\phi}_{\mathbb{M}_{\mathbf{L}}} = \emptyset$. Hence, $\val{\phi}_{\mathbb{M}_{\mathbf{L}}} = (\descr{\phi}_{\mathbb{M}_{\mathbf{L}}})^{[1]} = \emptyset^{[1]} = Z$, as required.

Likewise, one can show item 2 of the lemma under the additional assumption that  $\mathbf{L}$ derives $\phi\vdash \bot$.

Now, assuming that $\mathbf{L}$ derives neither $\top \vdash \phi$ nor $\phi \vdash \bot$, we proceed by induction on $\phi$.  The base cases are straightforward. Consider $\phi=\alpha \vee \beta$. Now 
\begin{center}
\begin{tabular}{rcll}
$\mathbb{M}_{\mathbf{L}}, z \Vdash \alpha \vee \beta \quad$  & iff & $\forall z' \in Z [zEz' \Rightarrow \mathbb{M}_{\mathbf{L}}, z' \not\succ \alpha \vee \beta]$ \\
& iff & $\forall z' \in Z [zEz' \Rightarrow (\mathbb{M}_{\mathbf{L}}, z' \not\succ \alpha \text{ or } \mathbb{M}_{\mathbf{L}}, z' \not\succ \beta )]$ \\
& iff & $\forall z' \in Z [F_z \cap J_{z'} = \emptyset \Rightarrow (\alpha \notin J_{z'} \text{ or } \beta\notin J_{z'})]$ & inductive hypothesis\\
& iff & $\forall z' \in Z [F_z \cap J_{z'}  \neq \emptyset \text{ or } (\alpha \notin J_{z'} \text{ or } \beta\notin J_{z'})].$
\end{tabular}
\end{center}
Consider $z' \in Z$ defined by $z'=(\lfloor \top \rfloor, \lceil(\alpha \vee \beta)\rceil)$, where $\lfloor \top \rfloor$ and $\lceil(\alpha \vee \beta)\rceil$ denote, respectively, the filter generated by $\top$ and the ideal generated by $\alpha \vee \beta$. The state $z'$ is indeed  well-defined since by assumption $(\alpha\vee\beta)\notin \lfloor \top \rfloor$. Moreover, since $\top \not\vdash \alpha \vee \beta$, this filter and ideal are disjoint.  Clearly $\alpha \in J_{z'}$ and $\beta \in J_{z'}$ so we must have $F_z \cap \lceil(\alpha \vee \beta)\rceil \neq \emptyset$ so $\alpha \vee \beta \in F_z$. 
Conversely, suppose $\alpha \vee \beta \in F_z$ and consider $z' \in Z$ with $zEz'$. Then 
$F_z \cap J_{z'}=\emptyset$ so $\alpha \vee \beta \notin J_{z'}$ and since this is a down-set we 
have $\alpha \notin J_{z'}$ and $\beta \notin J_{z'}$ and by the inductive hypothesis we have 
$\mathbb{M}_{\mathbf{L}}, z' \not\succ \alpha$ and $\mathbb{M}_{\mathbf{L}}, z' \not\succ \beta$. 

The proof that $\mathbb{M}_{\mathbf{L}}, z \succ \alpha \vee \beta$ iff $\alpha \vee \beta \in J_z$ follows easily 
from the fact that $J_z$ is an ideal. The proof of $\phi=\alpha \wedge \beta$ is similar to $\phi=\alpha \vee \beta$ 
but with the role of $\Vdash$ and $\succ$ interchanged.

Now consider $\phi = \Box\psi$ and assume that $\mathbb{M}_{\mathbf{L}}, z \Vdash \Box \psi$. We have
\begin{center}
\begin{tabular}{rcll}
$\mathbb{M}_{\mathbf{L}}, z \Vdash \Box\psi$ 
&  iff & $\forall z' \in Z [zR_{\Box}z' \Rightarrow \mathbb{M}_{\mathbf{L}}, z' \not\succ \psi]$\\
&  iff & $\forall z' \in Z [zR_{\Box}z' \Rightarrow \psi \notin J_{z'}]$ & inductive hypothesis\\
&  iff & $\forall z' \in Z [F_z \cap \Box J_{z'} = \emptyset \Rightarrow \psi \notin J_{z'}]$ \\
&  iff & $\forall z' \in Z [\psi \in J_{z'} \Rightarrow F_z \cap \Box J_{z'} \neq  \emptyset]$ \\
\end{tabular}
\end{center}
Consider $z'=(\lfloor \top \rfloor, \lceil\psi\rceil)$. Clearly $\psi \in J_{z'}$ so there exists $\alpha \in F_z \cap \Box J_{z'}$. Now $\alpha = \Box \beta$ for some $\beta \leq \psi$ (in the lattice order of $\mathbb{A}_{\mathbf{L}}$), i.e. $\beta \vdash \psi$ and therefore $\Box \beta \vdash \Box \psi$, whence $\Box \psi \in F_z$. For the converse, if $\Box \psi \in F_z$ then clearly the statement $\forall z' \in Z [F_z \cap \Box J_{z'} = \emptyset \Rightarrow \psi \notin J_{z'}]$ is true and so $\mathbb{M}_{\mathbf{L}},z \Vdash \Box \psi$. Now
\begin{center}
\begin{tabular}{rcll}
$\mathbb{M}_{\mathbf{L}}, z \succ \Box \psi$  & iff & $\forall \z' \in Z [\z' E\z  \Rightarrow \mathbb{M}_{\mathbf{L}},\z' \not\Vdash \Box \psi]$\\
& iff & $\forall \z' \in Z [\z' E\z  \Rightarrow \Box \psi \notin F_{z'}]$ & from above\\
& iff & $\forall \z' \in Z [F_{z'} \cap J_{z} = \emptyset\Rightarrow \Box \psi \notin F_{z'} ]$\\
& iff & $\forall \z' \in Z [\Box \psi \in F_{z'} \Rightarrow  F_{z'} \cap J_{z} \neq \emptyset]$\\
& iff & $\Box \psi \in J_z$. 
\end{tabular}
\end{center}
The forward implication of the last equivalence follows by taking $z'=(\lfloor\Box \psi\rfloor, \lceil\bot \rceil)$.

The case of $\phi=\Diamond \psi$ follows using a similar proof to that of $\phi=\Box \psi$ except starting by first showing
$\mathbb{M}_{\mathbf{L}},z\succ \Diamond \psi$ iff $\Diamond \psi \in J_z$.  
\end{proof}

\begin{theorem}
	The basic non-distributive modal logic $\mathbf{L}$ is  complete w.r.t.~the class of graph-based $\mathcal{L}$-frames.
\end{theorem}
\begin{proof} 
Consider an $\mathcal{L}$-sequent $\phi \vdash \psi$ that is not derivable in $\mathbf{L}$. Then $\lfloor \phi\rfloor\cap \lceil \psi\rceil  = \emptyset$ in the Lindenbaum-Tarski algebra. Let $z:= (\lfloor \phi\rfloor, \lceil \psi\rceil)$ be the corresponding state in $\mathbb{M}_{\mathbf{L}}$
 By Lemma~\ref{lemma:truthlemma} we have
$\mathbb{M},z\Vdash \phi$ and $\mathbb{M}, z\succ \psi$, but $zEz$. Hence $\mathbb{M} \not\models \phi \vdash \psi$.  
\end{proof}
\begin{remark}
The proof via canonical model given above is of course constructive; defining the canonical model as we do by taking disjoint filter-ideal pairs (rather than e.g.~maximally disjoint filter-ideal pairs) does not require any of the equivalents of Zorn's lemma.

Completeness can also be argued via canonical extension in the following way which does not make use of the truth lemma.  Firstly, we observe that Proposition \ref{Prop:Xplus:canonicalextension} can be readily extended to the statement that for any $\mathcal{L}$-algebra $\mathbb{A} = (\mathbb{L},  \Box, \Diamond)$, the complex algebra of its associated graph-based structure $\mathbb{F_A}$ is the canonical extension $\mathbb{A}^\delta$. Secondly,  we observe that  any graph-based structure validates exactly the sequents valid on its complex algebra (cf.~Lemma \ref{lemma:for:completeness}).

Hence, if the $\mathcal{L}$-sequent $\phi \vdash \psi$ is not derivable in $\mathbf{L}$, then by algebraic completeness, $\phi \vdash \psi$ is not valid in the Lindenbaum--Tarski algebra; then $\phi \vdash \psi$ is not valid in the canonical extension of the Lindenbaum--Tarski algebra, which, as discussed above, is the complex algebra of the canonical model; then (Lemma \ref{lemma:for:completeness}) $\phi \vdash \psi$ is not satisfied in the canonical model.
\end{remark}

\section{Sahlqvist correspondence on graph-based frames}
\label{sec:correspondence}
\paragraph{Parametric notions.} We find it useful to phrase the correspondence results of the present section in terms of a number of notions, parametric in $E$, which generalize familiar notions about sets and relations which are staples of correspondence theory in Kripke frames. 
The following definition will make it possible to concisely express relevant first order conditions. Properties of this definition are collected in  Section \ref{sec:appendix:properties of E-comp}.

\begin{definition}
\label{def:E-composition of relations}
For any graph $\mathbb{X} = (Z,E)$ and relations $R, S \subseteq Z \times Z$, the  $E$-{\em compositions} of $R$ and $S$ are the relations $R \circ_{E} S\subseteq Z\times Z$ and $R \bullet_{E} S\subseteq Z\times Z$ defined as follows: for any $a, x\in Z$,
	\[
	x (R \circ_{E} S) a \quad \text{iff} \quad \exists b (x R b \; \& \; E^{(1)}[b] \subseteq S^{(0)}[a]).
	\]
	\[
	a (R \bullet_{E} S) x \quad \text{iff} \quad \exists y(aR y \; \&\; E^{(0)}[ y]\subseteq  S^{(0)} [x]).
	\]

\end{definition}
If $E = \Delta$, then $E^{(1)}[b] = E^{(0)}[b] =  \{b\}$ for every $b\in Z$, and hence  $(R \circ_{E} S)$ and $(R \bullet_{E} S)$ reduce both to the usual relational composition of $R$ and $S$. The interpretation of  $E$-compositions will be discussed in Section \ref{sec: interpretation}, while a number of their key properties are proven in Appendix \ref{sec:appendix:properties of E-comp}.

\begin{definition}
\label{def:terminology}
For any graph $\mathbb{X} = (Z, E)$,  the relation $R\subseteq Z\times Z$ is:\\
\begin{tabular}{lllllllllllllllll}
$E$-{\em reflexive} & iff  & $E\subseteq R$; & $\, $& {\em sub}-$E$ &iff & $R\subseteq E$; & $\, $&$E_\circ$-{\em transitive} & iff & $R\circ_E R\subseteq R$; & $\, $&$E_\bullet$-{\em transitive} & iff & $R\bullet_E R\subseteq R$. 
\end{tabular}
\end{definition}
When $E: = \Delta$, we obtain the usual reflexivity, transitivity etc. 


\begin{proposition}
\label{lemma:correspondences}
For any graph-based $\mathcal{L}$-frame $\mathbb{F} = (\mathbb{X}, R_\Box, R_\Diamond)$,
{\small{
\begin{enumerate}
\item $\mathbb{F}\models \Box\phi\vdash \phi\quad $ iff $\quad E\subseteq R_\Box\quad$ ($R_\Box$ is $E$-reflexive).
\item $\mathbb{F}\models \phi\vdash \Diamond\phi\quad $ iff $\quad E\subseteq R_\blacksquare\quad $ ($R_\Diamond$ is $E$-reflexive).
\item $\mathbb{F}\models \Box\phi\vdash \Box\Box\phi\quad $ iff $\quad R_{\Box}\bullet_E R_{\Box}\subseteq R_{\Box}\quad$ ($R_\Box$ is $E_\bullet$-transitive).
\item $\mathbb{F}\models \Diamond\Diamond\phi\vdash \Diamond\phi\quad $ iff $\quad R_{\Diamond}\circ_E R_{\Diamond}\subseteq R_{\Diamond}\quad$($R_\Diamond$ is $E_\circ$-transitive). 
\item $\mathbb{F}\models \phi\vdash \Box \phi\quad $ iff $\quad R_\Box\subseteq E\quad$ ($R_\Box$ is sub-$E$).
\item $\mathbb{F}\models \Diamond\phi\vdash \phi\quad $ iff $\quad R_\blacksquare\subseteq E\quad$ ($R_\Diamond$ is sub-$E$)
\end{enumerate}
}}
\end{proposition}

\begin{proof}
The modal principles above are all Sahlqvist (cf.~\cite[Definition 3.5]{CoPa:non-dist}). Hence, they all have first-order correspondents, {\em both} on Kripke frames {\em and} on graph-based $\mathcal{L}$-frames, which can be computed e.g.~via the algorithm ALBA (cf.~\cite[Section 4]{CoPa:non-dist}). Below, we do so 
for the modal axiom in item 1 (for the remaining items, see Appendix \ref{sec:appendixproof of sahlqvist corrs}). In what follows, the variables $j$ are interpreted as elements of the set $J := \{(a^{[10]}, a^{[1]})\mid a\in Z\}$ which completely join-generates $\mathbb{F}^{+}$, and the variables $m$ as elements of $M: = \{(x^{[0]}, x^{[01]}) \mid x\in Z\}$ which completely meet-generates $\mathbb{F}^{+}$.  
{\small{
\begin{center}
	\begin{tabular}{r l l l}
		&$\forall p$  [$\Box p \leq p $]\\
		iff& $\forall p \forall j \forall m  [(j\le \Box p \ \&\  p\le m )\Rightarrow j\le m]$
		& first approximation\\
		iff & $ \forall j \forall m  [ j\le \Box m \Rightarrow j\le m]$
		&Ackermann's Lemma \\
		iff& $ \forall m  [\Box m\le m]$
		&$J$ completely join-generates $\mathbb{F}^{+}$
	\end{tabular}
\end{center}
		}}%
		Translating the universally quantified algebraic inequality above into its concrete representation in $\mathbb{F}^+$ requires using  the interpretation of $m$ as ranging in $M$ and the definition of $[R_{\Box}]$ and $[R_\blacksquare]$, as follows:
		{\small{
		\begin{center}
			\begin{tabular}{r l l l}
				&$\forall x\in Z$~~~~ $R_{\square}^{[0]}[x^{[01]}]\subseteq E^{[0]}[x]$& translation\\
				iff &$\forall x\in Z$~~~~ $R_{\square}^{[0]}[x]\subseteq E^{[0]}[x]$& Lemma \ref{equivalents of I-compatible} since $R_{\square}$ is $E$-compatible\\
				iff& $R^c_{\square}\subseteq E^c$& \eqref{eq:def:square brackets}\\
				iff& $E\subseteq R_{\square}$.& 
			\end{tabular}
		\end{center} }}
\end{proof}

\section{Graph-based frames as models of informational entropy}
\label{sec: interpretation}
As shown in the previous sections, graph-based frames -- such as those defined for the language $\mathcal{L}$ -- provide a mathematically grounded semantic environment for lattice-based logics such as $\mathbf{L}$. However, in order for this environment to `make sense' in a more fundamental way, we need to: (a) specify how it generalizes the Kripke semantics of classical normal modal logic; (b) couple it with an extra-mathematical interpretation which simultaneously accounts for the meaning of {\em all} connectives, and coherently extends to the meaning of axioms and of their first order correspondents.  Below, we propose a way to address these issues.

By assumption, the graphs $\mathbb{X} = (Z, E)$ on which the semantics of $\mathbf{L}$ is based are reflexive, i.e.~$\Delta\subseteq E$. Hence, a good starting point to address (a) is to understand this semantics when $E = \Delta$. In this case, the polarity arising from $\mathbb{X}$ is $\mathbb{P_X} = (Z_A, Z_X, I_{\Delta^c})$, and, as is well known and easy to see (cf.~\cite[Proposition 1]{roughconcepts}), the complete lattice $\mathbb{X}^+$ arising from $\mathbb{X}$ is (isomorphic to) the powerset algebra $\mathcal{P}(Z)$, and can be represented as a concept lattice the join-generators of which are $(a^{[10]}, a^{[1]}) = (\{a\}, \{a\}^c)$ for every $a\in Z$, and the meet generators of which are $(x^{[0]}, x^{[01]}) = (\{x\}^c, \{x\})$ for every $x\in Z$. Notice also that if $E: = \Delta$, then $B^{\uparrow} = B^c$ and $Y^{\downarrow} = Y^c$ for all $B, Y\subseteq Z$. Hence,  the  interpretation of $\mathcal{L}$-formulas on frames based on $\mathbb{X} = (Z, \Delta)$ reduces as shown below. These computations show that indeed, when $E: = \Delta$, we recover the usual Kripke-style interpretation of the logical connectives, both propositional and modal.

{\small{
\begin{center}
	\begin{tabular}{r c l c l l}
		$V(p)$ & $ = $ & $(\val{p}, \descr{p})$& $ = $ & $(\val{p}, \val{p}^c)$\\
		$V(\top)$ & $ = $ & $(Z, Z^{[1]})$ & $ = $ & $(Z, Z^c)$ \\
		$V(\bot)$ & $ = $ & $(Z^{[0]}, Z)$& $ = $ & $(Z^c, Z)$\\
		$V(\phi\wedge\psi)$ & $ = $ & $(\val{\phi}\cap \val{\psi}, (\val{\phi}\cap \val{\psi})^{[1]})$& $ = $ & $(\val{\phi}\cap \val{\psi}, (\val{\phi}\cap \val{\psi})^c)$\\
		$V(\phi\vee\psi)$ & $ = $ & $((\descr{\phi}\cap \descr{\psi})^{[0]}, \descr{\phi}\cap \descr{\psi})$& $ = $ & $(\val{\phi}\cup \val{\psi}, (\val{\phi}\cup \val{\psi})^c)$\\
		$V(\Box\phi)$ & $ = $ & $(R_{\Box}^{[0]}[\descr{\phi}], (R_{\Box}^{[0]}[\descr{\phi}])^{[1]})$& $ = $ & $((R_{\Box}^{-1}[\val{\phi}^c])^c, R_{\Box}^{-1}[\val{\phi}^c])$ & ($\ast$)\\
		$V(\Diamond\phi)$ & $ = $ & $((R_{\Diamond}^{[0]}[\val{\phi}])^{[0]}, R_{\Diamond}^{[0]}[\val{\phi}])$ & $ = $ & $(R_{\Diamond}^{-1}[\val{\phi}], (R_{\Diamond}^{-1}[\val{\phi}])^c)$& ($\ast\ast$)\\
	\end{tabular}
\end{center}
}}
To justify the lines marked with ($\ast$) and ($\ast\ast$),
{\small{
\begin{center}
\begin{tabular}{r c l l c r c ll}
$R_{\Box}^{[0]}[\descr{\phi}]$& =&$(R_{\Box}^c)^{(0)}[\val{\phi}^c]$ &  & $\quad$ & $R_{\Diamond}^{[0]}[\val{\phi}]$ & =&$(R_{\Diamond}^c)^{(0)}[\val{\phi}]$\\
& =& $\{z\mid \forall y(y\notin\val{\phi}\Rightarrow zR_{\Box}^c y)\} $&  & $\quad$ && =&  $\{z\mid \forall y(y\in\val{\phi}\Rightarrow zR_{\Diamond}^c y)\} $& \\
& =& $\{z\mid \forall y( zR_{\Box} y \Rightarrow y\in\val{\phi})\} $&  & $\quad$ && =& $\{z\mid \forall y(zR_{\Diamond} y\Rightarrow y\in\val{\phi}^c)\} $&  \\
& =& $(\{z\mid \exists y( zR_{\Box} y \ \&\  y\in\val{\phi}^c)\})^c $&  & $\quad$ && =& $(\{z\mid \exists y(zR_{\Diamond} y\ \&\  y\in\val{\phi})\})^c $&   \\
& =& $(R_{\Box}^{-1}[\val{\phi}^c])^c $ &  & $\quad$ & & =& $(R_{\Diamond}^{-1}[\val{\phi}])^c $ & \\
\end{tabular}
\end{center}
}}
Earlier on, we observed that the $E$-composition of relations reduces to the usual relational composition when $E: = \Delta$, and so do the `$E$-versions' of relational properties such as reflexivity and transitivity (cf.~Definition \ref{def:terminology}).  So, in a slogan, the graph-based interpretation of the modal operators is {\em classical modulo a shift} from $\Delta$ to $E$. In what follows we focus on this shift.

Drawing from the literature in information science and modal logic, we can regard the vertices of $\mathbb{X} = (Z, E)$ as states, and interpret $zEy$ as `$z$ is {\em indiscernible} from $y$'. The reflexivity of $E$ is the minimal property we assume of such a relation, i.e.~that every state is indiscernible from itself.\footnote{In well-known settings (e.g.~\cite{pawlak,halpern}), indiscernibility is modelled as an equivalence relation. However, transitivity will fail, for example, when $z Ey$ iff $d(z, y)<\alpha$ for some distance function $d$. It has been argued in the psychological literature (cf.~\cite{tversky1977features,nosofsky1991stimulus}) that symmetry will fail in situations where indiscernibility is understood as similarity, defined e.g.~as  $z$ is similar to $y$ iff $z$ has all the features $y$ has.} The closure $a^{[10]}$ of any $a\in Z$ arises by first considering the set $a^{[1]}$ of all the states from which $a$ is not indiscernible, and then the set of all the states that can be told apart from every state in $a^{[1]}$. Then clearly, $a$ is an element of $a^{[10]}$, but this is as far as we can go: $a^{[10]}$ represents a \emph{horizon} to the possibility of completely `knowing' $a$. This horizon could be epistemic, cognitive, technological, or evidential. Hence,  $E: =\Delta$ represents the limit case in which $a^{[10]} = \{a\}$ for each state, i.e.~there are no bounds to the `knowability' of each state of $Z$. 

As we saw in Definition~\ref{def:graph:based:frame:and:model}, the elements of the complex algebra of a graph-based frame are tuples $(B, Y)$ such that $Y = B^{[1]}$ and $B = Y^{[0]}$. This two-sided representation yields a corresponding  interpretation of $\mathcal{L}$-formulas $\varphi$ as tuples $(\val{\varphi}, \descr{\varphi})$ which, as discussed above, reduce to $(\val{\varphi}, \val{\varphi}^c)$  when $E: = \Delta$. Hence, formulas $\varphi$ are assigned both a {\em satisfaction set} $\val{\varphi}$ and a {\em refutation set} $\descr{\varphi}$ which, as is the case when $E: = \Delta$, determine each other, i.e.~$\descr{\varphi} = \val{\varphi}^{[1]}$ and $\val{\varphi} = \descr{\varphi}^{[0]}$.  The latter identities imply that $\val{\varphi}^{[10]} = \val{\varphi}$ and $\descr{\varphi}^{[01]} = \descr{\varphi}$, i.e.~both the satisfaction and the refutation set of any formula are {\em stable}. The stability requirement, which is mathematically justified by the need of defining a compositional semantics for $\mathcal{L}$, can also be understood at a more fundamental level: if $E$ encodes an {\em inherent boundary} to perfect knowability (i.e.~the {\em informational entropy} of the title), this boundary should be incorporated in the meaning of formulas which are both satisfied and refuted `up to $E$', i.e.~not by arbitrary subsets of the domain of the graph, but only by subsets which are preserved (i.e.~faithfully translated) in the shift from $\Delta$ to $E$. 

This  is similar to the {\em persistency} restriction in the interpretation of formulas of intuitionistic (modal) logic. Just like the interpretation of implication changes in the shift from classical to intuitionistic semantics, the interpretation of {\em disjunction} changes from classical to graph-based semantics and becomes {\em weaker}: the stipulation $\val{\phi\vee\psi} = (\descr{\phi}\cap \descr{\psi})^{[0]}$ requires a state $z$ to satisfy $\phi\vee\psi$ exactly when $z$ can be told apart from any state that refutes both $\phi$ and $\psi$. All states in $\val{\phi}\cup\val{\psi}$ will satisfy this requirement, but more states might as well which neither satisfy $\phi$ nor $\psi$, provided that $E$ detects their being different from every state that refutes both $\phi$ and $\psi$.

Additional relations on graphs-based frames can be regarded as encoding  {\em subjective indiscernibility}, i.e.~$zR_\Box y$ iff $z$ is indiscernible from $y$ according to a given agent.  Under this interpretation,
the stipulation $\val{\Box\phi} = R_{\Box}^{[0]}[\descr{\phi}]$ requires $\Box \phi$ to be satisfied at exactly those states that the agent can tell  apart from each state refuting $\phi$, and the stipulation $\descr{\Diamond\phi} = R_{\Diamond}^{[0]}[\val{\phi}]$ requires $\Diamond \phi$ to be refuted at exactly those states that the agent can tell  apart from each state satisfying $\phi$, and  be satisfied at the states that can be told apart from every state in $\descr{\Diamond\phi}$. 
Hence, under the interpretation indicated above, these semantic clauses  support the usual reading of   $\Box\phi$ as `the agent knows/believes $\phi$' and $\Diamond \phi$ as `the agent considers $\phi$ plausible'. 

Finally, we illustrate, by way of examples, how this interpretation coherently extends to axioms.
In Proposition \ref{lemma:correspondences}, we show that, also on graph-based frames, well known modal axioms from classical modal logic have first-order correspondents, which are the parametrized `$E$-counterparts' of the first order correspondents on Kripke frames. Interestingly, this surface similarity goes deeper, and in fact guarantees that the intended meaning of a given axiom under a given interpretation is preserved in the translation from $\Delta$ to $E$. As a first illustration of this phenomenon, consider the axiom $\Box\phi\vdash \phi$, which, under the epistemic reading,  in classical modal logic captures the characterizing property of the {\em factivity} of knowledge (if the agent knows $\phi$, then $\phi$ is true). This axiom corresponds to $E\subseteq R_\Box$ on graph-based frames (cf.~Proposition \ref{lemma:correspondences}). This condition requires that if the agent tells apart $z$ from $y$, then indeed $z$ is not indistinguishable from $y$. That is, the agent's assessments are correct, which {\em mutatis mutandis}, is exactly what factivity is about. 

Likewise, as is well known, under the epistemic reading, axiom $\Box\phi\vdash \Box\Box\phi$  captures the so called {\em positive introspection}  condition: knowledge of $\phi$ implies knowledge of knowing $\phi$. This axiom corresponds to $R_\Box\bullet_E R_\Box\subseteq R_\Box$ on graph-based frames (cf.~Proposition \ref{lemma:correspondences}). 
This condition requires that if the agent cannot distinguish a state $y$ from $a$ and nothing from which $y$ is (in principle) indistinguishable she can distinguish from $x$, then she cannot distinguish $x$ from $a$. Equivalently, if she can distinguish $x$ from $a$, then every state which she cannot distinguish from $a$ cannot be distinguished (in principle) from some state from which she can distinguish $x$. This is exactly what positive introspection  is about. 
 As a third example, consider the axiom $\phi\vdash \Box\phi$, which in the epistemic logic literature is referred to as the  {\em omniscience} principle (if $\phi$ is true, then the agent knows $\phi$). This axiom corresponds to $R_\Box\subseteq E$ on graph-based frames (cf.~Proposition \ref{lemma:correspondences}). This condition requires the agent to tell apart $z$ from every state $y$ from which $z$ is not indistinguishable, which is indeed what an omniscient agent should be able to do. 

\section{Sources of informational entropy} 
\label{sec:examples}


In this section we discuss two examples of the use graph-based models to capture situations where informational entropy arises. The first considers synonymy in natural a language while the second deals with colour perception an the limits of the human visual apparatus.

\paragraph{Synonymy in natural language.}

The exact nature of synonymy is debated, but there is evidence to suggest that this relation, although reflexive, can fail to be an equivalence, both on symmetry and transitivity. For example, one study \cite{chodorow1988tool} looks at English synonyms in an online thesaurus and finds high degree of asymmetry. For example, \url{http://thesaurus.com} lists \emph{cushion} in the entry for \emph{pillow}, but does not list \emph{pillow} in the entry for \emph{cushion}, suggesting that cushion is a synonym for pillow but not vice versa. To take another example, in a South African context, the term \emph{chips} covers both what Americans would call \emph{fries} and what the British would call \emph{crisps}. A South African English speaker would thus regards \emph{chips} as a synonym for both  \emph{fries} and \emph{crisps}, but would regard neither \emph{fries} nor \emph{crisps} as synonyms for \emph{chips}. \emph{Chips} is by far the most commonly used word, with \emph{fries} and \emph{crisps} only used when disambiguation is required. This can be modelled with the 
graph-based frame in the figure below, where the solid arrows represent the $E$-relation, taken as the South African synonymity relation. As the reader can easily verify, the closed sets of this graph are exactly $\varnothing$, $\{\text{\it fries} \}$, $\{\text{\it crisps} \}$, $\{\text{\it fries}, \text{\emph{crisps}}\}$ and $\{\text{\it crisps}, \text{\emph{chips}},\text{\it fries}\}$\footnote{Notice that since the $E$-relation in this example is only `one step', it is automatically transitive and therefore a pre-order. Hence, unsurprisingly, the associated concept lattice is distributive.}. For any given word, the smallest of these sets containing it can be thought of as its `semantic scope'.  In particular, this accurately represents the fact that the words \emph{fries} and \emph{crisps} have unambiguous meanings while, without the benefit of context, \emph{chips} could mean either of the others.     

Now consider an American tourist trying to make sense of local usage. Having some experience with British usage, she assumes \emph{chips} and \emph{fries} as interchangeable terms, and say she also knows that South Africans use \emph{chips} as a synonym for \emph{crisps}. This epistemic situation is modelled by the dashed arrows in the figure below which define the $E$-compatible relation $R$.

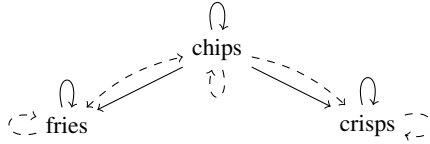
	\begin{center}
		\begin{tikzpicture}[scale=1]
		\node (fries)    at (-2,0)    {\footnotesize fries};
		\node (crisps)   at (2,0)    {\footnotesize crisps};
		\node (chips)    at (0,1)    {\footnotesize chips};
		\draw[->, loop above, min distance=5mm]  (fries) to node {} (fries);
		\draw[->, loop above, min distance=5mm]  (chips) to node {} (chips);
		\draw[->, loop above, min distance=5mm]  (crisps) to node {} (crisps);
		\draw[->, dashed, loop left, min distance=5mm]  (fries) to node {} (fries);
		\draw[->, dashed, loop below, min distance=5mm]  (chips) to node {} (chips);
		\draw[->, dashed, loop right, min distance=5mm]  (crisps) to node {} (crisps);
		\draw[<->, dashed, bend left=15] (fries) to node {} (chips);
		\draw[->, dashed, bend left=15] (chips) to node {} (crisps);
		\draw[<-] (fries) to node {} (chips);
		\draw[<-] (crisps) to node {} (chips);
		\end{tikzpicture}	
	\end{center}

We could evaluate a proposition letter $p$, with intended interpretation  `specific terms for fried potatoes', to  $(\val{p}, \descr{p}) = (\{\text{\it fries}, \text{\emph{crisps}}\}, \{\text{\emph{chips}}\})$, which would yield  $\val{\Box_{R} p} = \{\text{\it crisps}\}$ capturing the fact that \emph{crisps} is the only term the tourist can be sure denotes a specific kind of fried potato.

		


\paragraph{Perceptual limits.}
The wavelength of visible light lies roughly in the rage from 380 to 780 nanometres. The smallest difference between wavelengths in this range which is detectable by the human eye is known as the \emph{differentiation minimum}. The differentiation minimum varies with wavelengths and is best in the green-blue (around $490$ nm) and orange  (around $590$ nm) spectra, where it is as low as $1$ nm. It goes as high as $7$ nm in the low $400$ and middle $600$ ranges, but averages round $4$ nm over the spectrum of visible light. Deficient colour vision is characterized by significantly higher individual differentiation minima in certain ranges \cite{KrudyLadunga}.

We model this situation using a graph-based frame. Firstly, write $[380,780]$ for $\{ x \in \mathbb{N} \mid 380 \leq x  \leq 780 \}$ and represent the differentiation minimum by the function $\delta: [380,780] \rightarrow \mathbb{N}$ mapping every integer valued wavelength between $380$ nm and $780$ nm to the associated differentiation minimum. Represent the (possibly deficient) colour vision of an agent $A$ by $\delta_A: [380,780] \rightarrow \mathbb{N}$ such that $\delta_A(x) \geq \delta(x)$ for all $x \in [380,780]$. We will make the assumption that $\delta$ has no sudden ``jumps'', specifically, that for all $x \in [380,779]$, $|\delta(x) - \delta(x+1)| \leq 1$. We will assume that for all $x \in [380,780]$, if $(x - \delta_{A}(x)) \geq 380$, there exists $x_{\ell} \in [x  - \delta_{A}(x) + 1, x]$ such that $\delta(x_{\ell}) = x_{\ell} - (x - \delta_{A}(x))$ and, symmetrically, that if  $(x + \delta_{A}(x)) \leq 780$, there exists $x_{r} \in [x, x + \delta_{A}(x) - 1]$ such that $\delta(x_{r}) = (x + \delta_{A}(x)) - x_{r}$.  This assumption is needed for technical reasons. However, is justified in the case of $x_{\ell}$ (and symmetrically in the case of $x_r$) by the consideration that, since $x-\delta_A(x)$ is the first point to the left of $x$ in the spectrum which agent can discern from $x$, there should be a point in between $x-\delta_A(x) + 1$ and $x$ which is minimally discernible from $x-\delta_A(x)$ according to differentiation minimum (and could be $x$ itself, if the agent's perception at this point coincides with the differentiation minimum).

Let $\mathbb{F} = (\mathbb{X},  R_{\Diamond}, R_{\Box})$ where $\mathbb{X} = ([380,780],E)$ such that $xEy$ iff $|x - y| < \delta(x)$ and $x R_{\Diamond} y$ iff $x R_{\Box} y$ iff $|x - y| < \delta_A(x)$. Note that $E$ is reflexive, but need be neither symmetric nor transitive.  Using the assumptions above, one can prove that $R_{\Box}$ is $E$-compatible.

Suitable proposition letters to interpret on $\mathbb{F}$ would be colour terms like \emph{green}, \emph{yellow}, \emph{orange} etc. For example, according to the standard division of the spectrum into colours, one would evaluate $\val{\text{green}} = [520,560]$, $\val{\text{yellow}} = [560,590]$ and $\val{\text{orange}} = [590,635]$. As a simplified and stylized example (but one nevertheless not too unrealistic for the range we focus on subsequently), let us take $\delta$ and $\delta_{A}$ to be defined as in the following table:

\begin{center}
	\begin{tabular}{|c|c|c|}
		\hline
		Interval &$\delta$ &$\delta_{A}$\\
		\hline
		$370$ - $519$ &3 &7\\
		\hline
		$520$ - $550$ &4 &8\\
		\hline
		$551$ - $570$ &3 &7\\
		\hline
		$571$ - $780$ &2 &6\\
		\hline
	\end{tabular}
\end{center}

In this model we get $\val{\Box \text{green}} = R^{[0]}[\descr{\text{green}}] = R^{[0]}[\,[370,516] \cup [563,780] \, ] = [524,556]$ which represent the range of wavelengths that the agent definitely perceives as green. On the other hand $\descr{\Diamond \text{green}} = R^{[1]}[\val{\text{green}}] = R^{[1]}[[520,560]] = [370,512] \cup [567,780]$ which is the set of wavelengths which the agent definitely perceives as \emph{not} green. This leaves the intervals $[513,523]$ and $[557,568]$ where the agent cannot tell whether the corresponding colour is green or not.    




\section{Conclusions}

The present contributions lay the ground for a number of further developments, some of which are listed below.

{\bf Parametric Sahlqvist theory.} In Proposition \ref{lemma:correspondences} we were able to formulate our correspondence results as parametric versions (where $E$ is the parameter)  of well known relational properties such as reflexivity and transitivity (cf.~Definition \ref{def:terminology}).  This phenomenon was also observed in \cite[Proposition 5]{roughconcepts}.  A natural question is whether these instances can be subsumed by a more general and systematic parametric Sahlqvist theory, where the generalized frame correspondent of any Sahlqvist formula would be obtainable directly as a parametrization of its classical frame correspondents.

{\bf G\"odel-McKinsey-Tarski translation.} As mentioned in Section \ref{sec: interpretation}, one way of making sense of the present framework is by comparing it with the relational semantics of intuitionistic logic. In the later, the relation $E$ is reflexive and transitive,  and rather than being used to generate the semantics of modal operators on powerset algebras, it is used to generate an algebra of  stable sets, namely the persistent (i.e.~upward closed or downward closed) sets. Hence a natural direction is to build a non-distributive version of the transfer results induced by a suitable counterpart of G\"odel-McKinsey-Tarski translation. We are presently pursuing this direction.

{\bf Many-valued graph-based semantics.} In this paper, we only treat examples of informational entropy due to linguistic and perceptual limits. However, a very interesting area of application for this framework is the formal analysis of informational entropy induced by theoretical frameworks adopted to conduct scientific experiments. These situations are also amenable to be studied using a many-valued version of the present framework, which we have started to outline in \cite{graph-based-MV}.


\bibliography{ref}
\bibliographystyle{plain}	 
\appendix
\newpage
\section{Equivalent compatibility conditions in formal contexts}

\begin{lemma}\label{equivalents of I-compatible appendix}
	\begin{enumerate}
		\item The following are equivalent for every formal context $\mathbb{P} = (A, X, I)$ and every relation $R\subseteq A\times X$:
		\begin{enumerate}
			\item [(i)] $R^{(0)}[x]$ is Galois-stable for every $x\in X$;
			
			\item [(ii)]  $R^{(0)} [Y]$ is Galois-stable for every $Y\subseteq X$;
			\item [(iii)] $R^{(1)}[B]=R^{(1)}[B^{\uparrow\downarrow}]$ for every  $B\subseteq A$.
		\end{enumerate}
		\item The following are equivalent for every formal context $\mathbb{P} = (A, X, I)$ and every relation $R\subseteq A\times X$:

		\begin{enumerate}
			\item [(i)] $R^{(1)}[a]$ is Galois-stable for every $a\in A$;
			
			\item [(ii)]  $R^{(1)} [B]$ is Galois-stable, for every $B\subseteq A$;
			\item[(iii)] $R^{(0)}[Y]=R^{(0)}[Y^{\downarrow\uparrow}]$ for every $Y\subseteq X$.
		\end{enumerate}
	\end{enumerate}
\end{lemma}
\begin{proof} We only prove item 1, the proof of item 2 being similar. For $(i)\Rightarrow (ii)$, see \cite[Lemma 4]{Tarkpaper}. The converse direction is immediate.
	
	$(i)\Rightarrow (iii)$. Since $(\cdot)^{\uparrow\downarrow}$ is a closure operator, $B\subseteq B^{\uparrow\downarrow}$. Hence, Lemma \ref{lemma: basic}.1 implies that $R_{\Box}^{(1)}[B^{\uparrow\downarrow}]\subseteq R_{\Box}^{(1)}[B]$. For the converse inclusion, let $x\in R_{\Box}^{(1)}[B]$. By Lemma \ref{lemma: basic}.2, this is equivalent to $B\subseteq R_{\Box}^{(0)}[x]$. Since  $R_{\Box}^{(0)}[x]$ is Galois-stable by assumption, this implies that $B^{\uparrow\downarrow}\subseteq R_{\Box}^{(0)}[x]$, i.e., again by Lemma \ref{lemma: basic}.2, $x\in R_{\Box}^{(1)}[B^{\uparrow\downarrow}]$. This shows that $R_{\Box}^{(1)}[B]\subseteq R_{\Box}^{(1)}[B^{\uparrow\downarrow}]$, as required.

	$(iii)\Rightarrow (i)$. Let $x\in X$. It is enough to show that $(R_{\Box}^{(0)}[x])^{\uparrow\downarrow}\subseteq R_{\Box}^{(0)}[x]$.  By Lemma \ref{lemma: basic}.2, $R_\Box^{(0)}[x]\subseteq R_\Box^{(0)}[x]$ is equivalent to $x\in R_\Box^{(1)}[R_\Box^{(0)}[x]]$. By assumption, $R_{\Box}^{(1)}[R_\Box^{(0)}[x]]=R_\Box^{(1)}[(R_\Box^{(0)}[x])^{\uparrow\downarrow}]$, hence $x\in R_\Box^{(1)}[(R_\Box^{(0)}[x])^{\uparrow\downarrow}]$. Again by Lemma \ref{lemma: basic}.2, this is equivalent to $(R_{\Box}^{(0)}[x])^{\uparrow\downarrow}\subseteq R_{\Box}^{(0)}[x]$, as required.	
\end{proof}

\section{Composing relations on graph-based structures}
\label{sec:appendix:properties of E-comp}

The present section collects properties of the $E$-compositions  (cf.~Definition \ref{def:E-composition of relations}).

\begin{lemma}
\label{lemma: E-composition as composition of maps}
For any graph $\mathbb{X} = (Z,E)$, relations $R, S\subseteq Z \times Z$ and $a, x\in Z$,
	\begin{center}
		\begin{tabular}{lcl}
			$(R \circ_{E} S)^{[0]}[a]= R^{[0]}[E^{[0]}[S^{[0]}[a]]]$, &$\quad \quad$ &$(R \circ_{E} S)^{[1]}[x]= R^{[1]}[E^{[1]}[S^{[1]}[x]]]$,\\
			$(R \bullet_{E} S)^{[0]}[x]= R^{[0]}[E^{[1]}[S^{[0]}[x]]]$ &and  &$(R \bullet_{E} S)^{[1]}[a]= R^{[1]}[E^{[0]}[S^{[1]}[a]]]$.\\
		\end{tabular}
	\end{center}
\end{lemma}

\begin{proof} We only prove  the identities in the left column.
	\begin{center}
	\begin{tabular} {r c l l}
	$R^{[0]}[E^{[0]}[S^{[0]}[a]]]$ & = & $R^{[0]}[E^{[0]}[\{ x\mid xS^c a\}]]$ & definition of $S^{[0]}[a]$\\
	& = & $R^{[0]}[\{ b\mid \forall x(xS^c a\Rightarrow b E^c x)\}]$ & definition of $E^{[0]}[ -]$\\
	& = & $R^{[0]}[\{ b\mid S^{[0]}[ a]\subseteq  E^{[1]} [b]\}]$ \\
	& = & $R^{[0]}[\{ b\mid E^{(1)}[ b]\subseteq  S^{(0)} [a]\}]$ & Lemma \ref{lemma:round and square brackets} \\
	& = & $\{ x\mid \forall b(E^{(1)}[ b]\subseteq  S^{(0)} [a]\Rightarrow xR^c b)\}$& definition of $R^{[0]}[ -]$\\
	& = & $(\{ x\mid \exists b(xR b \; \&\; E^{(1)}[ b]\subseteq  S^{(0)} [a])\})^c$\\
	& = & $(\{ x\mid x  (R \circ_{E} S) a\})^c$ & Definition \ref{def:E-composition of relations}\\
	& = & $\{ x\mid x  (R \circ_{E} S)^c a\}$ \\
        & = & $(R \circ_{E} S)^{[0]}[a]$.\\
	\end{tabular}
	\end{center}

	\begin{center}
		\begin{tabular} {r c l l}
			$R^{[0]}[E^{[1]}[S^{[0]}[x]]]$ & = & $R^{[0]}[E^{[1]}[\{ a\mid aS^c x\}]]$ & definition of $S^{[0]}[x]$\\
			& = & $R^{[0]}[\{ y\mid \forall a(aS^c x\Rightarrow a E^c y)\}]$ & definition of $E^{[1]}[ -]$\\
			& = & $R^{[0]}[\{ y\mid S^{[0]}[ x]\subseteq  E^{[0]} [y]\}]$ \\
			& = & $R^{[0]}[\{ y\mid E^{(0)}[ y]\subseteq  S^{(0)} [x]\}]$ & Lemma \ref{lemma:round and square brackets} \\
			& = & $\{ b\mid \forall y(E^{(0)}[ y]\subseteq  S^{(0)} [x]\Rightarrow bR^c y)\}$& definition of $R^{[0]}[ -]$\\
			& = & $(\{ b\mid \exists y(bR y \; \&\; E^{(0)}[ y]\subseteq  S^{(0)} [x])\})^c$\\
			& = & $(\{ b\mid b  (R \bullet_{E} S) x\})^c$ & Definition \ref{def:E-composition of relations}\\
			& = & $\{ b\mid b  (R \bullet_{E} S)^c x\}$ \\
			& = & $(R \bullet_{E} S)^{[0]}[x]$.\\
		\end{tabular}
	\end{center}
	
\end{proof}
\begin{lemma}
	If $R, T\subseteq Z\times Z$ and $R$ is  $E$-compatible, then so are $R \circ_{E} T$ and $R \bullet_{E} T$.
\end{lemma}
\begin{proof}
	 Let $a\in Z$. By Lemma \ref{lemma: E-composition as composition of maps},  $(R\, ;T)^{[0]} [a] = R^{[0]}[I^{[0]}[T^{[0]} [a]]]$, hence the following chain of identities holds: \[((R\, ;T)^{[0]} [a])^{[01]} = (R^{[0]}[I^{[0]}[T^{[0]} [a]]])^{[01]} = R^{[0]}[I^{[0]}[T^{[0]} [a]]] = (R\, ;T)^{[0]} [a],\] the second identity in the chain above following from the $E$-compatibility of  $R$ and Lemma \ref{equivalents of I-compatible}.1.  The remaining conditions for the $E$-compatibility of $R \circ_{E} T$ and and $R \bullet_{E} T$ are shown similarly.
	 \end{proof}

The following lemma is the  counterpart   of  \cite[Lemma 6]{roughconcepts} in graph-based semantics.

\begin{lemma}\label{lemma:comp4}
	If  $R,T\subseteq Z\times Z$ are $E$-compatible, then  for any $B,Y\subseteq Z$, \[(R\circ_E T)^{[1]}[Y]=R^{[1]}[E^{[1]}[T^{[1]}[Y]]] \quad \quad (R \circ_E T)^{[0]}[B]=R^{[0]}[E^{[0]}[T^{[0]}[B]]].\]
	\[(R\bullet_E T)^{[1]}[B]=R^{[1]}[E^{[0]}[T^{[1]}[B]]] \quad \quad (R \bullet_E T)^{[0]}[Y]=R^{[0]}[E^{[1]}[T^{[0]}[Y]]].\]
        \end{lemma}
\begin{proof} We only prove the first identity, the remaining ones  being proved similarly.
	\begin{center}
		\begin{tabular}{r c l l}
			$R^{[1]}[E^{[1]}[T^{[1]}[Y]]]$ &
			= & $R^{[1]}[E^{[1]}[T^{[1]}[\bigcup_{x\in Y}\{x\}]]]$ \\
			& = &  $R^{[1]}[E^{[1]}[\bigcap_{x\in Y}T^{[1]}[x]]]$ & Lemma \ref{lemma: basic:square:brackets}.5  \\
			& = &  $R^{[1]}[E^{[1]}[\bigcap_{x\in Y}E^{[0]}[E^{[1]}[T^{[1]}[x]]]]]$ & $T$  is $E$-compatible  \\
			
			& = &  $R^{[1]}[E^{[1]}[E^{[0]}[\bigcup_{x\in Y}E^{[1]}[T^{[1]}[x]]]]]$ & Lemma \ref{lemma: basic:square:brackets}.5\\
			& = &  $R^{[1]}[\bigcup_{x\in Y}E^{[1]}[T^{[1]}[x]]] $ & Lemma \ref{equivalents of I-compatible} \\
			& = & $\bigcap_{x\in Y} R^{[1]}[E^{[1]}[T^{[1]}[x]]]$ & Lemma \ref{lemma: basic:square:brackets}.5\\
			& = & $\bigcap_{x\in Y} (R\circ_ET)^{[1]}[x]$ & Lemma  \ref{lemma: E-composition as composition of maps} \\
			& = & $ (R\circ_E T)^{[1]}[\bigcup_{x\in Y}\{x\}]$ & Lemma \ref{lemma: basic:square:brackets}.5 \\
			& = & $ (R\circ_ET)^{[1]}[Y].$
		\end{tabular}
	\end{center}
\end{proof}

\begin{lemma}
	 If $R, T, U\subseteq Z\times Z$ are  $E$-compatible, then   $(R\circ_E T)\circ_E U = R\circ_E (T\circ_E U)$ and $(R\bullet_E T)\bullet_E U = R\bullet_E (T\bullet_E U)$.
\end{lemma}
\begin{proof}
For every $x\in Z$, repeatedly applying Lemma \ref{lemma:comp4} we get: 
	\begin{center}
		\begin{tabular}{r c l}
			$(R\circ_E (T\circ_E U))^{[1]}[x]$ &$=$& $R^{[1]}[E^{[1]}[(T\circ_E U)^{[1]} [x]]]$ \\
			&$=$& $R^{[1]}[E^{[1]}[T^{[1]}[E^{[1]}[ U^{[1]} [x]]]]]$ \\
			&$=$& $(R\circ_E T)^{[1]}[E^{[1]}[ U^{[1]} [x]]]$ \\
			&$=$& $((R\circ_E T)\circ_E U)^{[1]} [x],$ \\
		\end{tabular}
	\end{center}
	which shows that  $x (R\circ_E (T\circ_E U))^c a$ iff $x ((R\circ_E T)\circ_E U)^ca$ for any $x, a\in Z$, and hence $x (R\circ_E (T\circ_E U)) a$ iff $x ((R\circ_E T)\circ_E U)a$ for any $x, a\in Z$, as required. The remaining statements are proven similarly.
\end{proof}

\section{Proof of Proposition \ref{lemma:correspondences}}
\label{sec:appendixproof of sahlqvist corrs}

\noindent 2.
		\begin{center}
			\begin{tabular}{r l l l}
				&$\forall p$  [$ p \leq \Diamond p $]
				\\
				iff& $\forall p \forall j \forall m  [(j\le p \ \&\  \Diamond p\le m )\Rightarrow j\le m]$
				& first approximation\\
                iff& $\forall p \forall j \forall m  [(j\le p \ \&\   p\le\blacksquare m )\Rightarrow j\le m]$
				& adjunction\\
				iff & $ \forall j \forall m  [ j\le\blacksquare m  \Rightarrow j\le m]$
				&Ackermann's Lemma \\
				iff& $ \forall m  [\blacksquare m\le  m]$
				&$J$ completely join-generates $\mathbb{F}^{+}$\\
				i.e &$\forall x\in Z$~~~~ $R_{\blacksquare}^{[0]}[x^{[0 1]}]\subseteq E^{[0]}[x]$& translation\\
                iff&$\forall x\in Z$~~~~ $R_{\blacksquare}^{[0]}[x]\subseteq E^{[0]}[x]$& Lemma \ref{equivalents of I-compatible} since $R_{\blacksquare}$ is $E$-compatible\\
                iff& $R^c_{\blacksquare}\subseteq E^c$& \eqref{eq:def:square brackets} \\
                iff& $E\subseteq R_{\blacksquare}$.& 

			\end{tabular}
		\end{center}
	
\noindent 3.
		\begin{center}
			\begin{tabular}{r l l l}
				&$\forall p$  [$\Box p \leq \Box \Box p $]\\ iff& $\forall p \forall j \forall m  [(j\le \Box p \ \&\   p\le m )\Rightarrow j\le \Box \Box m]$
				&first approximation\\
				iff & $\forall j \forall m  [ j\le  \Box m\Rightarrow j\le \Box\Box m]$
				&  Ackermann's Lemma \\
				iff& $ \forall m  [\Box m\le \Box \Box m]$
				&$J$ completely join-generates $\mathbb{F}^{+}$\\
				i.e &$\forall x\in Z$~~~~ $R_{\Box}^{[0]}[x^{[01]}] \subseteq R_{\Box}^{[0]}[E^{[1]}[R_{\square}^{[0]}[x^{[01]}]]]$ & translation\\
                iff&$\forall x\in Z$~~~~ $R_{\Box}^{[0]}[x] \subseteq R_{\Box}^{[0]}[E^{[1]}[R_{\square}^{[0]}[x]]]$ & Lemma \ref{equivalents of I-compatible} since $R_{\square}$ is $E$-compatible\\
                 iff&$\forall x\in Z$~~~~ $R_{\Box}^{[0]}[x] \subseteq (R_{\Box}\bullet_E R_\Box)^{[0]}[x]$ &Lemma \ref{lemma:comp4} \\
                iff& $R_\Box^c\subseteq (R_{\Box}\bullet_E R_\Box)^c  $& \eqref{eq:def:square brackets} \\
                  iff& $R_{\Box}\bullet_E R_\Box\subseteq R_\Box  $.&  \\
			\end{tabular}
		\end{center}
		
\noindent 4.
		\begin{center}
			\begin{tabular}{r l l l}
				&$\forall p$  [$\Diamond\Diamond p \leq  \Diamond p $]\\
				iff& $\forall p \forall j \forall m  [(j\le  p \ \&\  \Diamond p\le m )\Rightarrow \Diamond\Diamond j\le  m]$
				&first approximation\\
				iff & $\forall j \forall m  [ \Diamond j\le   m\Rightarrow \Diamond\Diamond j\le  m]$
				& Ackermann's Lemma \\
				iff& $ \forall j  [\Diamond\Diamond j\le  \Diamond j]$
				&$M$ completely meet-generates $\mathbb{F}^{+}$\\
				i.e &$\forall a\in Z$~~~~ $R_{\Diamond}^{[0]}[a^{[10]}] \subseteq R_{\Diamond}^{[0]}[(R_{\Diamond}^{[0]}[a^{[10]}])^{[0]}]$& translation\\
                iff &$\forall a\in Z$~~~~ $R_{\Diamond}^{[0]}[a] \subseteq R_{\Diamond}^{[0]}[(R_{\Diamond}^{[0]}[a])^{[0]}]$ & Lemma \ref{equivalents of I-compatible} since $R_{\Diamond}$ is $E$-compatible\\
                iff &$\forall a\in Z$~~~~ $R_{\Diamond}^{[0]}[a] \subseteq (R_{\Diamond}\circ R_{\Diamond})^{[0]}[a]$ & Lemma \ref{lemma:comp4}\\
                iff& $R_{\Diamond}^c \subseteq (R_{\Diamond}\circ_E R_{\Diamond})^c $& \eqref{eq:def:square brackets}\\
                iff& $R_{\Diamond}\circ_E R_{\Diamond} \subseteq R_{\Diamond} $.& 

			\end{tabular}
		\end{center}

\noindent 5.
		\begin{center}
			\begin{tabular}{r l l l}
				&$\forall p  [p\leq \Box p ]$\\
				iff& $\forall p \forall j \forall m  [(j\le  p \ \&\  p\le m )\Rightarrow j\le \Box m]$
				& first approximation\\
				iff & $ \forall j \forall m  [ j\le  m \Rightarrow j\le \Box m]$
				&Ackermann's Lemma \\
				iff& $ \forall m  [ m\le \Box m]$
				&$J$ completely join-generates $\mathbb{F}^{+}$\\
				i.e &$\forall x\in Z$~~~~ $ E^{[0]}[x]\subseteq R_{\square}^{[0]}[x^{[01]}]$& translation\\
                iff &$\forall x\in Z$~~~~ $ E^{[0]}[x]\subseteq R_{\square}^{[0]}[x]$& Lemma \ref{equivalents of I-compatible} since $R_{\square}$ is $E$-compatible\\
                iff& $E^c \subseteq R_{\square}^c$& \eqref{eq:def:square brackets}\\
                 iff& $R_{\square}\subseteq E$.& 

			\end{tabular}
		\end{center}
		
\noindent 6.
		\begin{center}
			\begin{tabular}{r l l l}
				&$\forall p$  [$\Diamond  p \leq  p $]
				\\
				iff& $\forall p \forall j \forall m  [(j\le p \ \&\   p\le m )\Rightarrow \Diamond j\le m]$
				& first approximation\\
				iff & $ \forall j \forall m  [ j\le m  \Rightarrow \Diamond j\le m]$
				&Ackermann's Lemma \\
                iff & $ \forall j \forall m  [ j\le m  \Rightarrow j\le \blacksquare m]$
				& adjunction \\
				iff& $ \forall m  [ m\le\blacksquare  m]$
				&$J$ completely join-generates $\mathbb{F}^{+}$\\
				i.e &$\forall x\in Z$~~~~ $ E^{[0]}[x]\subseteq R_{\blacksquare}^{[0]}[x^{[01]}]$& translation\\
                iff&$\forall x\in Z$~~~~ $ E^{[0]}[x]\subseteq R_{\blacksquare}^{[0]}[x]$& Lemma \ref{equivalents of I-compatible} since $R_{\blacksquare}$ is $E$-compatible\\
                 iff& $E^c \subseteq R_{\blacksquare}^c$& \eqref{eq:def:square brackets}\\
                 iff& $R_{\blacksquare}\subseteq E$.& 			
                 \end{tabular}
		\end{center}

\end{document}